\documentclass[authoryear,preprint,noinfoline]{imsart}
\setattribute{thebibliography}{size}{\small}
\RequirePackage[OT1]{fontenc}
\RequirePackage{amsthm,amsmath,amssymb,mathrsfs}
\RequirePackage[sort,round]{natbib}
\RequirePackage[breaklinks,colorlinks]{hyperref}
\RequirePackage{hypernat}

\hypersetup{colorlinks=false}

\usepackage{verbatim}
\usepackage{soul}
\usepackage{graphicx}
\usepackage{ctable}

\startlocaldefs
\DeclareMathOperator{\Var}{Var}
\DeclareMathOperator{\Cov}{Cov}
\DeclareMathOperator{\Cor}{Cor}

\newcommand{\floor}[1]{\lfloor #1 \rfloor}

\newcommand{\field}[1]{\mathbb{#1}}
\newcommand{\R}{\field{R}}

\newcommand{\Z}{\field{Z}}

\newcommand{\E}{\field{E}}

\newcommand{\h}[1]{\boldsymbol{#1}}

\newcommand{\spec}[1]{\lambda({#1})}

\newcommand{\dd}{{\,\mathrm{d}}}

\theoremstyle{plain}
\newtheorem{theorem}{Theorem}

\newtheorem{lemma}[theorem]{Lemma}

\theoremstyle{definition}

\theoremstyle{definition}
\newtheorem{remark}{Remark}

\makeatletter

\newcommand{\Rmnum}[1]{\expandafter\@slowromancap\romannumeral #1@}
\makeatother

\newcommand{\comments}[1]{}

\newcounter{regular} 
\newcounter{stable} 

\endlocaldefs

\begin{document}

\begin{frontmatter}


\title{Simultaneous Inference of Covariances}
\runtitle{Simultaneous Inference of Covariances}

\begin{aug}
\author{\fnms{Han}
  \snm{Xiao}\ead[label=e1]{hxiao@stat.rutgers.edu}}
\and
\author{\fnms{Wei Biao}
  \snm{Wu}\ead[label=e2]{wbwu@galton.uchicago.edu}}
\runauthor{H.~Xiao, W.B.~Wu}


\affiliation{Rutgers University \and University of Chicago}
\address{
501 Hill Center\\
110 Frelinghuysen Road\\
Piscataway, NJ 08854\\
\printead{e1}
}
\address{
Department of Statistics \\
5734 S.~University Ave\\
Chicago, IL  60637\\
\printead{e2}}
\end{aug}

\begin{abstract}
We consider asymptotic distributions of maximum deviations of sample covariance matrices, a fundamental problem in high-dimensional inference of covariances. Under mild dependence conditions on the entries of the data matrices, we
establish the Gumbel convergence of the maximum deviations. Our result substantially generalizes earlier ones where the entries are assumed to be independent and identically distributed, and it provides a theoretical foundation for high-dimensional simultaneous inference of covariances.
\end{abstract}

\begin{keyword}[class=AMS]
\kwd[Primary ]{62H15,62H10}
\kwd[; secondary ]{62E20}
\end{keyword}

\begin{keyword}
\kwd{Covariance matrix}
\kwd{high dimensional analysis}
\kwd{maximal deviation}
\kwd{tapering}
\kwd{test for bandedness}
\kwd{test for covariance structure}
\kwd{test for stationarity}
\end{keyword}



\end{frontmatter}
\date{}

\renewcommand{\labelenumi}{(\roman{enumi})}

\section{Introduction}

Let $\boldsymbol{X}_n=\left(X_{ij}\right)_{1\leq i \leq n, 1\leq j
\leq m}$ be a data matrix whose $n$ rows form independent samples
from some population distribution with mean vector $\h{\mu}_n$ and
covariance matrix $\Sigma_n$. High dimensional data increasingly
occur in modern statistical applications in biology, finance and
wireless communication, where the dimension $m$ may be comparable
to the number of observations $n$, or even much larger than $n$.
Therefore, it is necessary to study the asymptotic behavior of
statistics of $\boldsymbol{X}_n$ under the setting that $m=m_n$
grows to infinity as $n$ goes to infinity.

In many empirical examples, it is often assumed that $\Sigma_n=I_m$,
where $I_m$ is the $m\times m$ identity matrix, so it is important to
perform the test
\begin{equation}
  \label{eq:1.test}
  H_0:\; \Sigma_n=I_m
\end{equation}
before carrying out further estimation or inference procedures.
Due to high dimensionality, conventional tests often do not work
well or cannot be implemented. For example, when $m>n$, the
likelihood ratio test (LRT) cannot be used because the sample
covariance matrix is singular; and even when $m<n$, the LRT is
drifted to infinity and lead to many false rejections if $m$ is
also large \citep{bai:2009}. \cite{ledoit:2003} found that the
empirical distance test \citep{nagao:1973} is not consistent when
both $m$ and $n$ are large. The problem has been studied by
several authors under the ``large $n$, large $m$'' paradigm.
\cite{bai:2009} and \cite{ledoit:2003} proposed corrections to the
LRT and the empirical distance test respectively. Assuming that 
the population distribution is Gaussian with $\h{\mu}_n=0$,
\cite{johnstone:2001} used the largest eigenvalue of the sample
covariance matrix $\boldsymbol{X}_n^\top\boldsymbol{X}_n$ as the
test statistic, and proved that its limiting distribution follows
the Tracy-Widom law \citep{tracy:1994}. Here we use the
superscript $\top$ to denote the transpose of a matrix or a
vector. His work was extended to the non-Gaussian case by
\cite{soshnikov:2002} and \cite{peche:2009}, where they assumed
the entries of $\boldsymbol{X}_n$ are independent and identically
distributed (i.i.d.) with sub-Gaussian tails.

Let $x_{1},x_{2},\ldots,x_{m}$ be the $m$ columns of
$\boldsymbol{X}_n$. In practice, the entries of the mean vector
$\h{\mu}_n$ are often unknown, and are estimated by $\bar x_i =
(1/n)\sum_{k=1}^n X_{ki}$. Write $x_i-\bar x_i$ for the vector
$x_i-\bar x_i \h{1}_n$, where $\h{1}_n$ is the $n$-dimensional vector
with all entries being one. Let $\sigma_{ij} = {\rm Cov}(X_{1 i},
X_{1 j})$, $1 \le i, j \le m$, be the covariance function, namely,
the $(i, j)$th entry of $\Sigma_n$. The sample covariance between 
columns $x_{i}$ and $x_{j}$ is defined as
\begin{equation*}
  \hat\sigma_{ij} = \frac{1}{n}{(x_{i}-\bar x_{i})^\top(x_{j}-\bar x_{j})}.
\end{equation*}
In high-dimensional covariance inference, a fundamental problem is to establish an asymptotic distributional theory for the maximum deviation
\begin{equation*}
  M_n = \max_{1\leq i<j\leq m} |\hat\sigma_{ij}-\sigma_{ij}|.
\end{equation*}
With such a distributional theory, one can perform statistical inference for structures of covariance matrices. For example, one can use $M_n$ to test the null hypothesis $H_0:\;\Sigma_n = \Sigma^{(0)}$, where $\Sigma^{(0)}$ is a pre-specified matrix. Here the null hypothesis can be that the population distribution is a stationary process so that $\Sigma_n$ is Toeplitz, or that $\Sigma_n$ has a banded structure.

It is very challenging to derive an asymptotic theory for $M_n$ if we allow dependence among $X_{11}, \ldots, X_{1 m}$. Many of the earlier results assume that the entries of the data matrix $\boldsymbol{X}_n$ are i.i.d.. In this case $\sigma_{i j} = 0$ if $i \not= j$. \cite{jiang:2004} derived the asymptotic distribution of
\begin{equation*}
  L_n  = \max_{1\leq i<j\leq m} |\hat\sigma_{ij}|.
\end{equation*}

\begin{theorem}[\citealp{jiang:2004}]
  \label{thm:jiang}
  Suppose $X_{i,j},\;i,j=1,2,\ldots$ are independent and identically
  distributed as $\xi$ which has variance one. Suppose $\E
  |\xi|^{30-\epsilon}<\infty$ for any $\epsilon>0$. If
  $n/m\rightarrow c\in(0,\infty)$, then for any $y\in\R$,
  \begin{equation*}
    \lim_{n\rightarrow\infty}P\left(nL_n^2 - 4\log m + \log(\log m) + \log(8\pi) \leq y\right) = \exp\left(-e^{-y/2}\right).
  \end{equation*}
\end{theorem}
Jiang's work has attracted considerable attention, and been followed by
\cite{li:2010}, \cite{liu:2008}, \cite{zhou:2007} and \cite{li:2006}.
Under the same setup that $\boldsymbol{X}_n$ consists of
i.i.d. entries, these works focus on three directions (i) reduce the
moment condition; (ii) allow a wider range of $p$; and (iii) show that
some moment condition is necessary. In a recent article,
\cite{cai:2011a} extended those results in two ways: (i) the dimension
$p$ could grow exponentially as the sample size $n$ provided
exponential moment conditions; and (ii) they showed that the test
statistic $\max_{|i-j|>s_n} |\hat \sigma_{ij}|$ also converges to the
Gumbel distribution if each row of $\boldsymbol{X}_n$ is Gaussian
and is $s_n$-dependent. The latter generalization is important since
it is one of the very few results that allow dependent entries.

In this paper we shall show that a self-normalized version of $M_n$
converges to the Gumbel distribution under mild dependence conditions
on the vector $(X_{11}, \ldots, X_{1 m})$. Thus our result provides a
theoretical foundation for high-dimensional simultaneous inference of
covariances.

The rest of this article is organized as follows. We present the main
result in Section~\ref{sec:main}. In Section~\ref{sec:example}, we use
two examples on linear processes and nonlinear processes to
demonstrate that the technical conditions are easily satisfied. We
discuss three tests for the covariance structure using our main result
in Section~\ref{sec:cov_structure}. The proof is given in
Section~\ref{sec:proof}, and some auxiliary results are collected in
Section~\ref{sec:auxiliary}.

\section{Main result}
\label{sec:main} We consider a slightly more general situation
where population distribution can depend on $n$.  Let $\h{X}_n =
(X_{n,k,i})_{1\leq k\leq n, 1\leq i \leq m}$ be a data matrix
whose $n$ rows are i.i.d. $m$-dimensional random vectors with mean
$\h{\mu}_n=(\mu_{n,i})_{1\leq i \leq m}$ and covariance matrix
$\Sigma_n=(\sigma_{n,i,j})_{1\leq i,j \leq m}$. Let
$x_{1},x_{2},\ldots,x_{m}$ be the $m$ columns of
$\boldsymbol{X}_n$. Let $\bar x_i = (1/n)\sum_{k=1}^n X_{n,k,i}$,
and write $x_i-\bar x_i$ for the vector $x_i-\bar x_i \h{1}_n$.
The sample covariance between $x_{i}$ and $x_{j}$ is defined as
\begin{equation*}
  \hat\sigma_{n,i,j}
  = \frac{1}{n}{(x_{i}-\bar x_{i})^\top(x_{j}-\bar x_{j})}.
\end{equation*}
It is unnatural to study the maximum of a collection of random
variables which are on different scales, so we consider the
normalized version
$|\hat\sigma_{n,i,j}-\sigma_{n,i,j}|/\tau_{n,i,j}$, where
\begin{equation*}
  \tau_{n,i,j}
  = \Var\left[(X_{n,1,i}-\mu_{n,i})(X_{n,1,j}-\mu_{n,j})\right].
\end{equation*}
In practice, $\tau_{n,i,j}$ are usually unknown, and can be estimated by
\begin{equation*}
  \hat\tau_{n,i,j}
  = \frac{1}{n} \left|(x_{i}-\bar x_{i})\circ(x_{j}-\bar x_{j})
  -\hat\sigma_{n,i,j}\cdot \h{1}_n\right|^2.
\end{equation*}
where $\circ$ denotes the Hadamard product defined as $A\circ
B:=(a_{ij}b_{ij})$ for two matrices $A=(a_{ij})$ and $B=(b_{ij})$ with
the same dimensions. We thus consider
\begin{equation}
  \label{eq:max_cov}
  M_n = \max_{1 \leq i<j \leq m} \frac{|\hat\sigma_{n,i,j}-\sigma_{n,i,j}|}{\sqrt{\hat\tau_{n,i,j}}}.
\end{equation}
Due to the normalization procedure, we can assume without loss of
generality that $\sigma_{n,i,i}=1$ and $\mu_{n,i}=0$ for each $1\leq i
\leq m$.

Define the index set $\mathcal{I}_n=\{(i,j):\,1\leq i<j \leq m\}$, and
for $\alpha=(i,j)\in\mathcal{I}_n$, let
$X_{n,\alpha}:=X_{n,1,i}X_{n,1,j}$. Define
\begin{align*}
  & \mathcal{K}_n(t,p)
   = \sup_{1\leq i \leq m} \E \exp\left(t |X_{n,1,i}|^p\right), \\
  & \mathcal{M}_n(p)
   = \sup_{1\leq i \leq m} \E (|X_{n,1,i}|^{p}), \\
  & \tau_n = \inf_{1 \leq i<j \leq m} \tau_{n,i,j}, \\
  & \gamma_n = \sup_{\alpha, \beta \in\mathcal{I}_n
  \hbox{ {\tiny and} } \alpha\neq\beta}
  \left|\Cor(X_{n,\alpha},X_{n,\beta})\right|,  \\
  & \gamma_n(b)=\sup_{\alpha\in\mathcal{I}_n}
  \sup_{\mathcal{A}\subset\mathcal{I}_n,|\mathcal{A}|=b}
  \inf_{\beta \in \mathcal{A}}
    \left|\Cor(X_{n,\alpha},X_{n,\beta})\right|.
\end{align*}
We need the following technical conditions.
\begin{align*}
  & (\h{\mathrm{A1}}).\quad \liminf_{n \rightarrow\infty} \tau_n > 0. \\
  & (\h{\rm A2}).\quad \limsup_{n} \gamma_n<1. \\
  & (\h{\rm A3}).\quad \gamma_n(b_n) \cdot (\log b_n) = o(1)
  \hbox{ for any sequence $(b_n)$ such that $b_n\rightarrow\infty$.}\\
  & (\h{\rm A3'}).\quad \gamma_n(b_n) = o(1)
  \hbox{ for any sequence $(b_n)$ such that $b_n\rightarrow\infty$, and}\\
  & \qquad\qquad \sum_{\alpha,\beta\in\mathcal{I}_n}
  \left[\Cov(X_{n,\alpha},X_{n,\beta})\right]^2 = O(m^{4-\epsilon})
  \hbox{for some constant $\epsilon>0$}. \\
  & (\h{\rm A4}).\quad \log m = o\left(n^{p/(4+2p)}\right) \hbox{ and
  }
  \limsup_{n\rightarrow\infty}\mathcal{K}_n(t,p) < \infty \hbox{ for some constants} \\
  & \qquad \quad\;\;\;\hbox{$t>0$ and $0<p\leq 4$}. \\
  & (\h{\rm A4'}).\quad m=O(n^q) \hbox{ and }
  \limsup_{n\rightarrow\infty}\mathcal{M}_n(4q+4+\delta)<\infty \hbox{
    for some constants}\\
  & \qquad \quad\;\;\;\hbox{$q>0$ and $\delta>0$}.
\end{align*}
The two conditions ($\mathrm{A3}$) and ($\mathrm{A3'}$) require that
the dependence among $X_{n,\alpha},\;\alpha\in\mathcal{I}_n$, are not
too strong. They are translations of (B1) and (B2) in
Section~\ref{sec:normal_comparison} (see Remark~\ref{rk:empirical} for
some equivalent versions), and either of them will make our results
valid. We use (A2) to get rid of the case where they may be lots of
pairs $(\alpha,\beta)\in\mathcal{I}_n$ such that $X_{n,\alpha}$ and
$X_{n,\beta}$ are perfectly correlated. Assumptions
($\mathrm{A4}$) and ($\mathrm{A4'}$) connect the growth speed of $m$
relative to $n$ and the moment conditions. They are typical 
in the context of high dimensional covariance matrix estimation. 
Condition (A1) excludes the case that $X_{n,\alpha}$ is a constant.

\begin{theorem}
\label{thm:max_cov}
Suppose that $\h{X}_n = (X_{n,k,i})_{1\leq k\leq n, 1\leq i \leq m}$ 
is a data matrix whose $n$ rows are i.i.d. $m$-dimensional random
vectors, and whose entries have mean zero and variance one. Assume
(A1), (A2), either of ($\mathrm{A3}$) and ($\mathrm{A3'}$), and
either of ($\mathrm{A4}$) and ($\mathrm{A4'}$), then for any $y \in \R$,
  \begin{equation*}
    \lim_{n\rightarrow\infty}
    P\left(nM_n^2 - 4\log m + \log(\log m)
    + \log(8\pi) \leq y\right)
    = \exp\left(-e^{-y/2}\right).
  \end{equation*}
\end{theorem}

\section{Examples}
\label{sec:example}

Except for ($\mathrm{A4}$) and ($\mathrm{A4'}$), which put
conditions on every single entry of the random vector
$(X_{n,1,i})_{1\leq i \leq m}$, all the other conditions of
Theorem~\ref{thm:max_cov} are related to the dependence among
these entries, which can be arbitrarily complicated. In this
section we shall provide examples which satisfy the four conditions
(A1), (A2), ($\mathrm{A3}$) and ($\mathrm{A3'}$). Observe that if
each row of $\h{X}_n$ is a random vector with uncorrelated entries
(specifically, the entries are independent), then all these 
conditions are automatically satisfied. They are also satisfied
if the number of non-zero covariances is bounded.

\subsection{Stationary Processes}
Suppose $(X_{n,k,i})=(X_{k,i})$, and each row of $(X_{k,i})_{1\leq i
  \leq m}$ is distributed as a stationary process $(X_i)_{1\leq i \leq
  m}$ of the form
\begin{align*}
  X_i=g(\epsilon_i,\epsilon_{i-1},\ldots)
\end{align*}
where $\epsilon_i$'s are i.i.d. random variables, and $g$ is a
measurable function such that $X_i$ is well-defined. Let
$(\epsilon_i')_{i\in\Z}$ be an i.i.d. copy of $(\epsilon_i)_{i\in\Z}$,
and
$X_i'=g(\epsilon_i,\ldots,\epsilon_1,\epsilon_0',\epsilon_{-1},\epsilon_{-2},\ldots)$. Following
\cite{wu:2005}, define the {\it physical dependence measure} of order
$p$ by
\begin{align*}
  \delta_p(i)=\|X_i-X_i'\|_p.
\end{align*}
Define the squared tail sum
\begin{align*}
  \Psi_p(k)=\left[\sum_{j=k}^\infty (\delta_p(i))^2\right]^{1/2},
\end{align*}
and use $\Psi_p$ as a shorthand for $\Psi_p(0)$.

We give sufficient conditions for (A1), (A2), ($\mathrm{A3}$) and
($\mathrm{A3'}$) in the following lemma and leave its proof to the
supplementary file.
\begin{lemma}
  \label{thm:stationary}
  \begin{itemize}
  \item [(i)] If $0<\Psi_4<\infty$ and $\Var(X_iX_j)>0$ for
  all $i,j\in\Z$, then (A1) holds.
  \item [(ii)] If in addition, $|\Cor(X_iX_j,X_kX_l)|<1$ for all
    $i,j,k,l$ such that they are not all the same, then (A2) holds.
  \item [(iii)] Assume that the conditions of (i) and (ii) hold. If
    $\Psi_p(k)=o(1/\log k)$ as $k\rightarrow\infty$, then
    ($\mathrm{A3}$) holds. If $\sum_{j=0}^m (\Psi_4(j))^2 =
    O(m^{1-\delta})$ for some $\delta>0$, then ($\mathrm{A3'}$) holds.
  \end{itemize}
\end{lemma}

\begin{remark}
Let $g$ be a linear function with $g(\epsilon_i, \epsilon_{i-1},
\ldots) = \sum_{j=0}^\infty a_j \epsilon_{i-j}$, where
$\epsilon_j$ are i.i.d. with mean $0$ and $\E (|\epsilon_j|^p) <
\infty$ and $a_j$ are real coefficients with $\sum_{j=0}^\infty
a_j^2 < \infty$. Then the physical dependence measure $\delta_p(i)
= |a_i| \|\epsilon_0 - \epsilon_0'\|_p$. If $a_i = i^{-\beta}
\ell(i)$, where $1/2 < \beta < 1$ and $\ell$ is a slowly varying
function, then $(X_i)$ is a long memory process. Smaller $\beta$
indicates stronger dependence. Condition (iii) holds for all
$\beta \in (1/2, 1)$. Moreover, if $a_i = i^{-1/2}
(\log(i))^{-2}$, $i \ge 2$, which corresponds to the extremal case
with very strong dependence $\beta = 1/2$, we also have $\Psi_p(k)
= O( (\log k)^{-3/2} ) = o(1/\log k)$. So our dependence conditions 
are actually quite mild.

If $(X_i)$ is a linear process which is not identically zero, then
the following regularity conditions are automatically satisfied:
$\Psi_4>0$, $\Var(X_iX_j)>0$ for all $i,j\in\Z$, and $|\Cor(X_i
X_j, X_k X_l)|<1$ for all $i,j,k,l$ such that they are not all the
same.
\end{remark}

\subsection{Non-stationary Linear Processes}
Assume that each row of $(X_{n,k,i})$ is distributed as $(X_{n,i})_{1\leq i
  \leq m}$, which is of the form
\begin{align*}
  X_{n,i}=\sum_{t\in\Z} f_{n,i,t}\epsilon_{i-t},
\end{align*}
where $\epsilon_i,\,i\in\Z$ are i.i.d. random variables with mean
zero, variance one and finite fourth moment, and the sequence
$(f_{n,i,t})$ satisfies $\sum_{t\in \Z} f_{n,i,t}^2=1$. Denote by
$\kappa_4$ the fourth cumulant of $\epsilon_0$. For $1\leq i,j,k,l
\leq m$, we have
\begin{align*}
  \sigma_{n,i,j} & = \sum_{t\in\Z} f_{n,i,i-t}f_{n,j,j-t}, \\
  \Cov(X_{n,i}X_{n,j},X_{n,k}X_{n,l})
  & = \mathrm{Cum}(X_{n,i},X_{n,j},X_{n,k},X_{n,l})
  + \sigma_{n,i,k}\sigma_{n,j,l}+\sigma_{n,i,l}\sigma_{n,j,k},
\end{align*}
where $\mathrm{Cum}(X_{n,i},X_{n,j},X_{n,k},X_{n,l})$ is the fourth
order joint cumulant of the random vector
$(X_{n,i},X_{n,j},X_{n,k},X_{n,l})^\top$, which can be expressed as
\begin{align*}
  \mathrm{Cum}(X_{n,i},X_{n,j},X_{n,k},X_{n,l})
  = \sum_{t\in\Z} f_{n,i,i-t}f_{n,j,j-t}f_{n,k,k-t}f_{n,l,l-t}\kappa_4,
\end{align*}
by the multilinearity of cumulants. In particular, we have
\begin{align*}
  \Var(X_iX_j) = 1 + \sigma_{n,i,j}^2
  + \kappa_4\cdot\sum_{t \in \Z} f_{n,i,t}^2f_{n,j,t}^2.
\end{align*}
Since $\kappa_4 = \Var(\epsilon_0^2) - 2 \left(\E
\epsilon_0^2\right)^2 \geq -2$, the condition
\begin{align}
  \label{eq:cumulant}
  \kappa_4>-2
\end{align}
guarantees ($\mathrm{A1}$) in view of
\begin{align*}
  \Var(X_iX_j) \geq (1 + \sigma_{n,i,j}^2)(1+\min\{\kappa/2,0\})
  \geq \min\{1,1+\kappa/2\}>0.
\end{align*}

To ensure the validity of ($\mathrm{A2}$), it is natural to assume
that no pairs $X_{n,i}$ and $X_{n,j}$ are strongly correlated, {\it
  i.e.}
\begin{align}
  \label{eq:corr}
  \limsup_{n\rightarrow\infty}\sup_{1\leq i<j\leq m}\left|\sum_{t\in\Z}f_{n,i,i-t}f_{n,j,j-t}\right|<1.
\end{align}
We need the following lemma, whose proof is elementary and will be
given in the supplementary file.
\begin{lemma}
  \label{thm:normal}
  The condition (\ref{eq:corr}) suffices for ($\mathrm{A2}$) if
  $\epsilon_i$'s are i.i.d. $N(0,1)$.
\end{lemma}
As an immediate consequence, when $\epsilon_i$'s are i.i.d. $N(0,1)$,
we have
\begin{align*}
  \ell:=\limsup_{n\rightarrow\infty} \inf_{\ast}\inf_{\rho\in\R}\Var\left(X_{n,i}X_{n,j}-\rho X_{n,k}X_{n,l}\right) > 0,
\end{align*}
where $\inf_{\ast}$ is taken over all $1\leq i,j,k,l \leq m$ such that
$i<j$, $k<l$ and $(i,j)\neq(k,l)$. Observe that when $\epsilon_i$'s
are i.i.d. $N(0,1)$,
\begin{align}
  \label{eq:var}
  \Var\left(X_{n,i}X_{n,j}-\rho X_{n,k}X_{n,l}\right) & =
  2\cdot\sum_{t\in\Z}(f_{n,i,i-t}f_{n,j,j-t}-\rho f_{n,k,k-t}f_{n,l,l-t})^2\\ \nonumber
  & + \sum_{s<t}
    \left(f_{n,i,i-t}f_{n,j,j-s}+f_{n,i,i-s}f_{n,j,j-t} \right.\\ \nonumber
    & \quad\;
    \left. -\rho f_{n,k,k-t}f_{n,l,l-s}-\rho f_{n,k,k-s}f_{n,l,l-t}\right)^2;
\end{align}
and when $\epsilon_i$'s are arbitrary variables, the variance is given
by the same formula with the number 2 in (\ref{eq:var}) being replaced
by $2+\kappa_4$. Therefore, if (\ref{eq:cumulant}) holds, then
\begin{align*}
  \limsup_{n\rightarrow\infty} \inf_{\ast}\inf_{\rho\in\R}\Var\left(X_{n,i}X_{n,j}-\rho X_{n,k}X_{n,l}\right)
  \geq \min\{1,1+\kappa_4/2\}\cdot\ell>0,
\end{align*}
which implies $(\mathrm{A2})$ holds. To summarize, we have shown that
(\ref{eq:cumulant}) and (\ref{eq:corr}) suffice for $(\mathrm{A2})$.

Now we turn to Conditions ($\mathrm{A3}$) and ($\mathrm{A3'}$).
Set
\begin{align*}
  h_n(k)=\sup_{1\leq i \leq m}\left(\sum_{|t|=\floor{k/2}}^\infty f_{n,i,t}^2\right)^{1/2},
\end{align*}
where $\floor{x}=\max\{y\in\Z\,:\,y \leq x\}$ for any $x\in\E$, then we
have
\begin{align*}
  |\sigma_{n,i,j}| \leq 2h_n(0)h_n(|i-j|) = 2h_n(|i-j|).
\end{align*}
Fixing a subset $\{i,j\}$, for any integer $b>0$, there are at most
$8b^2$ subsets $\{k,l\}$ such that $\{k,l\}\subset B(i;b)\cup B(j;b)$,
where $B(x;r)$ is the open ball $\{y:|x-y|<r\}$. For all other subsets
$\{k,l\}$, we have
\begin{align*}
  |\Cov(X_{n,i}X_{n,j},X_{n,k}X_{n,l})| \leq (4+2\kappa_4)h_n(b),
\end{align*}
and hence ($\mathrm{A3}$) holds if we assume $h_n(k_n)\log k_n = o(1)$
for any positive sequence $(k_n)$ such that
$k_n\rightarrow\infty$. ($\mathrm{A3'}$) holds if we assume
\begin{align*}
  \sum_{k=1}^m [h_n(k)]^2 = O\left(m^{1-\delta}\right).
\end{align*}
for some $\delta>0$, because
\begin{align*}
  \left|\Cov(X_{n,i}X_{n,j},X_{n,k}X_{n,l})\right| \leq 2\kappa_4h_n(|i-j|) + 2h_n(|i-k|) + 2h_n(|i-l|).
\end{align*}

\section{Testing for covariance structures}
\label{sec:cov_structure} The asymptotic distribution given in
Theorem~\ref{thm:max_cov} has several statistical applications. One of
them is in high dimensional covariance matrix regularization, because
Theorem~\ref{thm:max_cov} implies a uniform convergence rate for all
sample covariances. Recently, \cite{cai:2011b} explored this
direction, and proposed a thresholding procedure for sparse covariance
matrix estimation, which is adaptive to the variability of each
individual entry. Their method is superior to the uniform thresholding
approach studied by \cite{bickel:2008b}.

Testing structures of covariance matrices is also a very important
statistical problem. As mentioned in the introduction, when the
data dimension is high, conventional tests often cannot be
implemented or do not work well. Let $\Sigma_n$ and $R_n$ be the
covariance matrix and correlation matrix of the random vector
$(X_{n,1,i})_{1\leq i \leq m}$ respectively. Two types of tests
have been studied under the large $n$, large $m$ paradigm.
\cite{chen:2010}, \cite{bai:2009}, \cite{ledoit:2003} and
\cite{johnstone:2001} considered the test
\begin{equation}
  \label{eq:3.identity}
  H_0:\;\Sigma_n=I_m;
\end{equation}
and \cite{liu:2008}, \cite{schott:2005}, \cite{srivastava:2005} and
\cite{jiang:2004} studied the problem of testing for complete
independence
\begin{equation}
  \label{eq:3.sphericity}
  H_0:\;R_n=I_m.
\end{equation}
Their testing procedures are all based on the critical assumption that
the entries of the data matrix $\h{X}_n$ are i.i.d., while the
hypotheses themselves only require the entries of $(X_{n,1,i})_{1\leq
  i \leq m}$ to be uncorrelated. Evidently, we can use $M_n$ in
(\ref{eq:max_cov}) to test (\ref{eq:3.sphericity}), and we only
require the uncorrelatedness for the validity of the limiting
distribution established in Theorem~\ref{thm:max_cov}, as long as the
mild conditions of the theorem are satisfied. On the other hand, we
can also take the sample variances into consideration, and use the
following test statistic
\begin{equation*}
  M_n' = \max_{1 \leq i\leq j \leq m} \frac{|\hat\sigma_{n,i,j}-\sigma_{n,i,j}|}{\sqrt{\hat\tau_{n,i,j}}}.
\end{equation*}
to test the identity hypothesis (\ref{eq:3.identity}), where
$\sigma_{n,i,j}=I\{i=j\}$. It is not difficult to verify that $M_n'$
has the same asymptotic distribution as $M_n$ under the same
conditions with the only difference being that we now have to take
sample variances into account as well, namely, the index set
$\mathcal{I}_n$ in Section~\ref{sec:main} is redefined as
$\mathcal{I}_n = \{(i,j):\,1\leq i\leq j\leq m\}$. Clearly, we can
also use $M_n'$ to test $H_0:\;\Sigma_n=\Sigma^0$ for some known
covariance matrix $\Sigma^0$.

By checking the proof of Theorem~\ref{thm:max_cov}, it can be seen
that if instead of taking the maximum over the set
$\mathcal{I}_n=\{(i,j):\,1\leq i<j \leq m\}$, we only take the maximum
over some subset $A_n\subset\mathcal{I}_n$ whose cardinality $|A_n|$
converges to infinity, then the maximum also has the Gumbel type
convergence with normalization constants which are functions of the
cardinality of the set $A_n$. Based on this observation, we are able
to consider three more testing problems.

\subsection{Test for stationarity}

Suppose we want to test whether the population is a stationary time
series. Under the null hypothesis, each row of the data matrix
$\h{X}_n$ is distributed as a stationary process $(X_i)_{1\leq i \leq
  m}$. Let $\gamma_l=\Cov(X_0,X_l)$ be the autocovariance at lag
$l$. In principle, we can use the following test statistic
\begin{equation*}
  \tilde T_n = \max_{1 \leq i\leq j \leq m} \frac{|\hat\sigma_{n,i,j}-\gamma_{i-j}|}{\sqrt{\hat\tau_{n,i,j}}}.
\end{equation*}
The problem is that $\gamma_l$ are unknown. Fortunately, they can not
only be estimated, but also be estimated with higher accuracy
\begin{equation*}
  \hat{\gamma}_{n,l} = \frac{1}{nm}\sum_{k=1}^n\sum_{i=|l|+1}^n (X_{n,k,i-|l|}-\hat\mu_n)(X_{n,k,i}-\hat\mu_n),
\end{equation*}
where $\hat\mu_n = (1/nm) \sum_{k=1}^n\sum_{i=1}^m X_{n,k,i}$, and we
are lead to the test statistic
\begin{equation*}
  T_n = \max_{1 \leq i\leq j \leq m} \frac{|\hat\sigma_{n,i,j}-\hat\gamma_{i-j}|}{\sqrt{\hat\tau_{n,i,j}}}.
\end{equation*}
Using similar arguments of Theorem~2 of \cite{wu:2011a}, under
suitable conditions, we have
\begin{equation*}
  \max_{0 \leq l \leq m-1} |\hat\gamma_{n,l}-\gamma_l| = O_P(\sqrt{\log m/nm}).
\end{equation*}
Therefore, the limiting distribution for $M_n$ in
Theorem~\ref{thm:max_cov} also holds for $T_n$.

\subsection{Test for bandedness}

In time series and longitudinal data analysis, it can be of interest
to test whether $\Sigma_m$ has the banded structure. The hypothesis to
be tested is
\begin{equation}
  \label{eq:3.bandedness}
  H_0:\; \sigma_{n,i,j}=0 \hbox{ if } |i-j|>B,
\end{equation}
where $B=B_n$ may depend on $n$. \cite{cai:2011a} studied this problem
under the assumption that each row of the data matrix $\h{X}_n$ is a
Gaussian random vector. They proposed to use the maximum sample
correlation outside the band
\begin{equation*}
  \tilde T_n = \max_{|i-j|>B} \frac{\hat\sigma_{n,i,j}}{\sqrt{\hat\sigma_{n,i,i}\hat\sigma_{n,j,j}}}
\end{equation*}
as the test statistic, and proved that $T_n$ also has the Gumbel type
convergence provided that $B_n=o(m)$ and several other technical
conditions hold.

Apparently, our Theorem~\ref{thm:max_cov} can be employed to test
(\ref{eq:3.bandedness}). If all the conditions of the theorem are
satisfied, the test statistic
\begin{equation*}
  T_n = \max_{|i-j|>B_n} \frac{|\hat\sigma_{n,i,j}|}{\sqrt{\hat\tau_{n,i,j}}}.
\end{equation*}
has the same asymptotic distribution as $M_n$ as long as
$B_n=o(m)$. Our theory does not need the normality assumption.

\subsection{Assess the tapering procedure}

Banding and tapering are commonly used regularization procedures in
high dimensional covariance matrix estimation. Convergence rates were
first obtained by \cite{bickel:2008a}, and later on improved by
\cite{cai:2010}. Let us introduce a weaker version of the latter
result. Suppose each row of $\h{X}_n$ is distributed as the random
vector $X=(X_i)_{1\leq i \leq m}$ with mean $\mu$ and covariance
matrix $\Sigma=(\sigma_{ij})$. Let $K_0,K$ and $t$ be positive
constants, and $\mathscr{C}_{\eta}(K_0, K, t)$ be the class of
$m$-dimensional distributions which satisfy the following conditions
\begin{align}
  & \max_{|i-j|=k} |\sigma_{ij}| \leq Kk^{-(1+\eta)} \quad\hbox{for all } k; \label{eq:3.decay} \\
  & \lambda_{\max}(\Sigma) \leq K_0; \cr
  & P\left[|v^\top(X-\mu)|>x\right] \leq e^{-tx^2/2} \quad \hbox{for all $x>0$ and $\|v\|=1$}; \nonumber
\end{align}
where $\lambda_{\max}(\Sigma)$ is the largest eigenvalue of
$\Sigma$. For a given even integer $1\leq B \leq m$, define the
tapered estimate of the covariance matrix $\Sigma$
\begin{align*}
  \hat\Sigma_{n,B_n} = \left(w_{ij}\hat\sigma_{n,i,j}\right),
\end{align*}
where the weights correspond to a flat top kernel and are given by
\begin{align*}
  w_{ij} = \left\{
    \begin{array}{ll}
      1, & \hbox{when } |i-j| \leq B_n/2, \\
      2-2|i-j|/B_n, & \hbox{when } B_n/2<|i-j|\leq B_n, \\
      0, & \hbox{otherwise}.
    \end{array}\right.
\end{align*}
\begin{theorem}[\citealp{cai:2010}]
  \label{thm:cai}
  If $m \geq n^{1/(2\eta+1)}$, $\log m = o(n)$ and
  $B_n=n^{1/(2\eta+1)}$, then there exists a constant $C>0$ such that
  \begin{align*}
    \sup_{\mathscr{C}_\eta} \E \left[\spec{\hat{\Sigma}_{n,B_n}-\Sigma}\right]^2
    \leq Cn^{-2\eta/(2\eta+1)} + C\frac{\log m}{n}.
  \end{align*}
\end{theorem}

We see that it is the parameter $\eta$ that decides the convergence
rate under the operator norm. After such a tapering procedure has been
applied, it is important to ask whether it is appropriate, and in
particular, whether (\ref{eq:3.decay}) is satisfied. We propose to use
\begin{align*}
  T_n = \max_{|i-j|>B_n} \frac{|\hat\sigma_{n,i,j}|}{\sqrt{\hat\tau_{n,i,j}}}
\end{align*}
as the test statistic. According to the observation made at the
beginning of Section~\ref{sec:cov_structure}, if the conditions of
Theorem~\ref{thm:max_cov} are satisfied, then
\begin{align*}
  T_n'= \max_{|i-j|>B_n} \frac{|\hat\sigma_{n,i,j}-\sigma_{i,j}|}{\sqrt{\hat\tau_{n,i,j}}}
\end{align*}
has the same limiting law as $M_n$. On the other hand,
(\ref{eq:3.decay}) implies that
\begin{align*}
  \max_{|i-j|>B_n} |\sigma_{i,j}| = O\left(n^{-(1+\eta)/(2\eta+1)}\right),
\end{align*}
so $T_n$ has the same limiting distribution as $T_n'$ if we further
assume $\log m = o\left(n^{2/(4\eta+2)}\right)$.

\section{Proof}
\label{sec:proof}


The proofs of Theorem~\ref{thm:max_cov} under ($\mathrm{A4}$) and
($\mathrm{A4'}$) are very similar, and they share a common Poisson
approximation step, which we will formulate in
Section~\ref{sec:max_mean} under a more general context, where the
limiting distribution of the maximum of sample means is
obtained. Since the proof under ($\mathrm{A4'}$) is more involved, we
provide the detailed proof under this assumption in
Section~\ref{sec:a4'}, and point out in Section~\ref{sec:a4} how it
can be adapted to give a proof under ($\mathrm{A4}$).

\subsection{Maximum of Sample Means: An Intermediate Step}
\label{sec:max_mean}

In this section we provide a general result on the maximum of sample
means. Let $\h{Y}_n=(Y_{n,k,i})_{1\leq k \leq n,\,i\in\mathcal{I}_n}$
be a data matrix whose $n$ rows are independent and identically
distributed, and whose entries have mean zero and variance one, where
$\mathcal{I}_n$ is an index set with cardinality
$|\mathcal{I}_n|=s_n$. For each $i \in \mathcal{I}_n$, let $y_i$ be
the $i$-th column of $\h{Y}_n$, $\bar y_i = (1/n)\sum_{k=1}^n
Y_{n,k,i}$.
Define
\begin{equation}
  \label{eq:max_mean}
  W_n = \max_{i \in \mathcal{I}_n} {|\bar y_i|}.
\end{equation}
Let $\Sigma_n$ be the covariance matrix of the $s_n$-dimensional
random vector $(Y_{n,1,i})_{i \in \mathcal{I}_n}$. 

\begin{lemma}
  \label{thm:max_mean}
  Assume $\Sigma_n$ satisfies either (B1) or (B2) of
  Section~\ref{sec:normal_comparison} and $\log s_n
  =o(n^{1/3})$. Suppose there is a constant $C>0$ such that $Y_{n,k,i}
  \in \mathscr{B}(1,Ct_n)$ for each $1\leq k \leq
  n,\;i\in\mathcal{I}_n$, with
  \begin{equation*}
    t_n = \frac{\sqrt{n}\delta_n}{(\log s_n)^{3/2}}, 
  \end{equation*}
  where $(\delta_n)$ is a sequence of positive numbers such that
  $\delta_n=o(1)$ and $(\log s_n)^3/n=o(\delta_n)$, and the definition
  of the collection $\mathscr{B}(d,\tau)$ is given in
  (\ref{eq:bernstein}). Then
  \begin{equation}
    \label{eq:max_mean_convergence}
    \lim_{n\rightarrow\infty}P\left(nW_n^2 - 2\log s_n + \log (\log s_n) + \log \pi \leq z \right)
    =\exp\left(-e^{-z/2}\right).
  \end{equation}
\end{lemma}

\begin{proof}
  For each $z\in\R$, let $z_n=a_{2s_n}z/2 + b_{2s_n}$. Let
  $(Z_{n,i})_{i\in\mathcal{I}_n}$ be a mean zero normal random vector
  with covariance matrix $\Sigma_n$.
  For any subset $A=\{i_1,i_2,\ldots,i_d\} \subset \mathcal{I}_n$, let
  $y_A=\sqrt{n}(\bar y_{i_1}, \bar y_{i_2}, \ldots, \bar
  y_{i_d})^\top$ and $Z_{A}=(Z_{i_1}, Z_{i_2}, \ldots, Z_{i_d})$. By
  Lemma~\ref{thm:zaitsev}, we have for $\theta_n =
  \delta_n^{1/2}/\sqrt{\log s_n}$ that
  \begin{align*}
    P\left(|y_{A}|_\bullet > z_n\right) & \leq P(|Z_A|_{\bullet}>z_n - \theta_n)
    + C_d\exp\left\{-\frac{\theta_n}{C_d \delta_n(\log s_n)^{-3/2}}\right\} \\
    & \leq P(|Z_A|_{\bullet}>z_n - \theta_n) + C_d\exp\left\{-(\log s_n)\delta_n^{-1/2}\right\}
  \end{align*}
  Therefore,
  \begin{align*}
    \sum_{A \subset \mathcal{I}_n, |A|=d} & P\left(|y_{A}|_\bullet > z_n\right) \\
    & \leq \sum_{A \subset \mathcal{I}_n, |A|=d} P(|Z_A|_{\bullet}>z_n - \theta_n)
    + C_ds_n^d \exp\left\{-(\log s_n)\delta_n^{-1/2}\right\}.
  \end{align*}
  Similarly, we have
  \begin{align*}
    \sum_{A \subset \mathcal{I}_n, |A|=d} & P\left(|y_{A}|_\bullet > z_n\right) \\
    & \geq \sum_{A \subset \mathcal{I}_n, |A|=d} P(|Z_A|_{\bullet}>z_n + \theta_n)
    - C_ds_n^d \exp\left\{-(\log s_n)\delta_n^{-1/2}\right\}.
  \end{align*}
  Since $(z_n\pm\theta_n)^2 = 2\log s_n - \log (\log s_n) - \log \pi +
  z + o(1)$, by Lemma~\ref{thm:normal_comparison}, we know
  \begin{align*}
    \lim_{n\rightarrow\infty} \sum_{A \subset \mathcal{I}_n, |A|=d} P(|Z_A|_{\bullet}>z_n \pm \theta_n)
    = \frac{e^{-dz/2}}{d\,!},
  \end{align*}
  and hence
  \begin{align*}
    \lim_{n\rightarrow\infty}\sum_{A \subset \mathcal{I}_n, |A|=d} P\left(|y_{A}|_\bullet > z_n\right)
    = \frac{e^{-dz/2}}{d\,!}.
  \end{align*}
  The proof is complete in view of Lemma~\ref{thm:poisson}.
\end{proof}

\subsection{Proof under ($\mathrm{A4'}$)}
\label{sec:a4'}

We divide the proof into three steps. The first one is a
truncation step, which will make the Gaussian approximation result
Lemma~\ref{thm:zaitsev} and the Bernstein inequality applicable,
so that we can prove Theorem~\ref{thm:max_cov} under the
assumption that all the involved mean and variance parameters are
known. In the next two steps we show that plugging in estimated
mean and variance parameters does not change the limiting
distribution.

\paragraph{\underline{\textnormal{\em Step 1: Truncation}}}

For notational simplicity we let $q=p/(4+2p)$. Define
\begin{align}
  \label{eq:3.truncation}
  \tilde X_{n,k,i} = X_{n,k,i}I\left\{|X_{n,k,i}|\leq n^{1/(4+2p)}\right\},
\end{align}
and define $\tilde M_n$ similarly as $M_n$ with $X_{n,k,i}$ being
replaced by its truncated version $\tilde X_{n,k,i}$. Since $\log m =
o (n^{q})$, we have
\begin{align*}
  P\left(\tilde M_n \neq M_n\right) & \leq \sum_{k=1}^n\sum_{i=1}^m
  P\left[|X_{n,k,i}| > n^{1/(4+2p)}\right] \\
  & \leq nm \mathcal{K}_n(t,p) \exp\left\{-t n^{p/(4+2p)}\right\} \\
  & = \mathcal{K}_n(t,p) \exp\left\{-tn^{q} + \log m + \log n\right\}
  = o(1).
\end{align*}
Therefore, in the rest of the proof, it suffices to consider $\tilde
X_{n,k,i}$. For notational simplicity, we still use $\tilde X_{n,k,i}$
to denote its centered version with mean zero.

Define $\tilde \sigma_{n,i,j}=\E \left(\tilde X_{n,1,i}\tilde
  X_{n,1,j}\right)$, and $\tilde \tau_{n,i,j} = \Var\left(\tilde
X_{n,1,i}\tilde X_{n,1,j}\right)$. Set
\begin{align*}
  M_{n,1} & = \max_{1\leq i<j \leq m} \frac{1}{\sqrt{\tilde \tau_{n,i,j}}}
  \left|\frac{1}{n} \sum_{k=1}^n \tilde X_{n,k,i}\tilde X_{n,k,j} -\tilde \sigma_{n,i,j} \right|; \\
  M_{n,2} & = \max_{1\leq i<j \leq m} \frac{1}{\sqrt{\tilde \tau_{n,i,j}}}
  \left|\frac{1}{n} \sum_{k=1}^n \tilde X_{n,k,i}\tilde X_{n,k,j} - \sigma_{n,i,j} \right|. \\
\end{align*}
Elementary calculations show that
\begin{align}
  \label{eq:3.2}
  \max_{1\leq i \leq j \leq m} |\tilde \sigma_{n,i,j} - \sigma_{n,i,j}| & \leq C \exp \left\{-tn^{q}/2
    \right\},\quad\hbox{and} \\ \label{eq:3.3}
  \max_{\alpha, \beta \in \mathcal{I}_n} \left|\Cov(\tilde X_{n,\alpha}, \tilde X_{n,\beta} ) -
  \Cov(X_{n,\alpha}, X_{n,\beta})\right| & \leq C \exp \left\{-tn^{q}/2
    \right\}.
\end{align}
By (\ref{eq:3.3}), we know the covariance matrix of $(\tilde
X_{n,\alpha})_{\alpha\in\mathcal{I}_n}$ satisfies either (B1) or (B2)
if $\Sigma_n$ satisfies (B1) or (B2) correspondingly. On the other
hand, we have by elementary calculation that there exists a constant
$C_{p}>0$ such that
\begin{align*}
  \limsup_{n\rightarrow\infty}\max_{\alpha \in \mathcal{I}_n} \E \exp\{C_p t  |\tilde X_{n,\alpha}|^{p/2}\} < \infty.
\end{align*}
It follows that when $0<p<2$, for each integer $r \geq 3$
\begin{align*}
  \E |\tilde X_{n,\alpha}|^{r} & \leq \E |\tilde X_{n,\alpha}|^{rp/2} \cdot \left(4 n^{2/(4+2p)}\right)^{r(1-p/2)} \\
  & \leq \left(4 n^{2/(4+2p)}\right)^{r(1-p/2)} r! (C_pt)^{-r} \E \exp\{C_p t |X_{n,\alpha}|^{p/2}\}.
\end{align*}
Therefore,
\begin{align*}
  \E_0\tilde X_{n,\alpha} \in \mathscr{B}\left[1,C\frac{\sqrt{n}}{n^{2p/(4+2p)}}\right].
\end{align*}
When $2 \leq p \leq 4$, it is easily seen that $\E_0\tilde X_{n,\alpha}
\in \mathscr{B}(1,C)$.  Since $\log m = o (n^q)$, we know all the
conditions of Lemma~\ref{thm:max_mean} are satisfied, and hence
\begin{equation}
  \label{eq:3.4}
  \lim_{n\rightarrow\infty}P\left(nM_{n,1}^2 - 4\log m + \log(\log m) + \log(8\pi) \leq y\right)
  = \exp\left(-e^{-y/2}\right).
\end{equation}
Combining (\ref{eq:3.2}) and (\ref{eq:3.3}), we know the preceding
equation (\ref{eq:3.4}) also holds with $M_{n,1}$ being replaced by
$M_{n,2}$.

\paragraph{\underline{\textnormal{\em Step 2: Effect of Estimated Means}}}
Set $\bar X_{n,i} = (1/n)\sum_{k=1}^n \tilde X_{n,k,i}$. Define
\begin{align*}
  M_{n,3} = \max_{1\leq i<j \leq m} \frac{1}{\sqrt{\tilde \tau_{n,i,j}}}
  \left|\frac{1}{n} \sum_{k=1}^n (\tilde X_{n,k,i}-\bar X_{n,i})(\tilde X_{n,k,j}-\bar X_{n,j}) - \sigma_{n,i,j} \right|.
\end{align*}
In this step we show that (\ref{eq:3.4}) also holds for
$M_{n,3}$. Observe that
\begin{align*}
  \left|M_{n,3} - M_{n,2}\right| \leq \max_{1\leq i<j \leq m} \frac{|\bar X_{n,i}\bar{X}_{n,j}|}{\sqrt{\tilde \tau_{n,i,j}}}
  \leq \max_{1 \leq i \leq m}|\bar X_{n,i}|^2 \cdot \left(\min_{1 \leq i<j \leq m} \tilde \tau_{n,i,j}\right)^{-1/2}.
\end{align*}
Since each $X_{n,k,i}$ is bounded by $2n^{1/(4+2p)}$, by Bernstein's
inequality we have for any constant $K>0$,
\begin{align*}
  \max_{1\leq i \leq m}P\left(|\bar X_{n,i}| > 2K\sqrt{\log m \over n}\right)
  & \leq C \exp\left\{ -\frac{2K^2 n \log m}{C n + 2K \sqrt{n\log m}
      \cdot 2n^{1/(4+2p)}} \right\} \\
  & \leq C m^{-K^2/C},
\end{align*}
and hence
\begin{align}
  \label{eq:3.1}
  \max_{1 \leq i \leq m}|\bar X_{n,i}| = O_P\left(\sqrt{\frac{\log m}{n}}\right),
\end{align}
which together with (\ref{eq:3.3}) implies that
\begin{align*}
  \left|M_{n,3} - M_{n,2}\right| = O_P\left(\frac{\log m}{n}\right) = o_P\left(\sqrt{\frac{1}{n \log m}}\right).
\end{align*}
Therefore, (\ref{eq:3.4}) also holds for $M_{n,3}$.

\paragraph{\underline{\textnormal{\em Step 3: Effect of Estimated Variances}}}
Denote by $\check{\sigma}_{n,i,j}$ the estimate of $\tilde\sigma_{n,i,j}$
\begin{align*}
  {\check\sigma_{n,i,j}} = \frac{1}{n} \sum_{k=1}^n (\tilde X_{n,k,i}-\bar X_{n,i})(\tilde X_{n,k,j}-\bar X_{n,j}).
\end{align*}
In the definition of $\tilde M_n$, $\tilde\tau_{n,i,j}$ is unknown,
and is estimated by
\begin{align*}
  {\check\tau_{n,i,j}} = \frac{1}{n}\sum_{k=1}^n
  \left[(\tilde X_{n,k,i} - \bar X_{n,i}) (\tilde X_{n,k,j} - \bar X_{n,j}) - {\check\sigma_{n,i,j}}\right]^2
\end{align*}
In this step we show that (\ref{eq:3.4}) holds for  $\tilde M_{n}$.
Since
\begin{align*}
  n\left|M_{n,3}^2-\tilde M_n^2\right| \leq nM_{n,3}^2 \cdot \max_{1\leq i<j\leq m}|1-\tilde\tau_{n,i,j}/\check\tau_{n,i,j}|,
\end{align*}
it suffices to show that
\begin{align}
  \label{eq:3.5}
  \max_{1\leq i<j\leq m} \left|\check\tau_{n,i,j}-\tilde\tau_{n,i,j}\right| = o_P(1/\log m).
\end{align}
Set
\begin{align*}
  \check\tau_{n,i,j,1} & = \frac{1}{n}\sum_{k=1}^n
  \left[(\tilde X_{n,k,i} - \bar X_{n,i}) (\tilde X_{n,k,j} - \bar X_{n,j}) - {\tilde\sigma_{n,i,j}}\right]^2 \\
    \check\tau_{n,i,j,2} & = \frac{1}{n}\sum_{k=1}^n
  \left(\tilde X_{n,k,i} \tilde X_{n,k,j} - {\tilde\sigma_{n,i,j}}\right)^2.
\end{align*}
Observe that
\begin{align*}
  \check\tau_{n,i,j,1} - \check\tau_{n,i,j} = (\check\sigma_{n,i,j} - \tilde\sigma_{n,i,j})^2
\end{align*}
which in together with (\ref{eq:3.4}) implies that
\begin{align}
  \label{eq:3.6}
  \max_{1\leq i<j\leq m} \left|\check\tau_{n,i,j,1}-\check\tau_{n,i,j}\right| = O_P \left(\log m/n\right).
\end{align}
Note that $\tilde X_{n,k,i,j}$ are uniformly bounded according to
the truncation (\ref{eq:3.truncation}), so
\begin{align*}
  \left(\tilde X_{n,k,i} \tilde X_{n,k,j} - {\tilde\sigma_{n,i,j}}\right)^2 \leq 64 n^{4/(4+2p)}.
\end{align*}
By Bernstein's inequality, we have
\begin{align*}
  \max_{1\leq i<j\leq m} P\left(|\check\tau_{n,i,j,2}-\tilde\tau_{n,i,j}| \geq  2n^{-q}\right)
  & \leq \exp\left\{-\frac{2 n^{2(1-q)}}{Cn + 2n^{1-q} \cdot 128n^{4/(4+2p)}/3}\right\} \\
  & \leq \exp\left(-n^q/100\right),
\end{align*}
and it follows that
\begin{align}
  \label{eq:3.7}
  \max_{1\leq i<j\leq m} \left|\check\tau_{n,i,j,2}-\tilde\tau_{n,i,j}\right| 
  = O_P(n^{-q}).
\end{align}
In view of (\ref{eq:3.6}), (\ref{eq:3.7}), and the assumption $\log
m=o(n^q)$, we know to show (\ref{eq:3.5}), it remains to prove
\begin{align}
  \label{eq:3.8}
  \max_{1\leq i<j\leq m} \left|\check\tau_{n,i,j,1}-\check\tau_{n,i,j,2}\right| = o_P(1/\log m).
\end{align}
Elementary calculations show that
\begin{align*}
  \max_{1\leq i<j\leq m} \left|\check\tau_{n,i,j,1}-\check\tau_{n,i,j,2}\right|
  \leq 4h_{n,1}^2h_{n,2} +  3h_{n,1}^4 + 4h_{n,4}^{1/2}h_{n,2}^{1/2}h_{n,1} 
  + 2h_{n,3}h_{n,1}^2,
\end{align*}
where
\begin{align*}
  h_{n,1} & =\max_{1\leq i \leq m} |\bar X_{n,i}| \\
  h_{n,2} & = \max_{1\leq i \leq m} \frac{1}{n}\sum_{k=1}^n \tilde X_{n,k,i}^2 \\
  h_{n,3} & = \max_{1\leq i\leq j \leq m}
  \left|\frac{1}{n}\sum_{k=1}^n \tilde X_{n,k,i} \tilde X_{n,k,j} - \tilde\sigma_{n,i,j}\right|\\
  h_{n,4} & = \check\tau_{n,i,j,2}.
\end{align*}
By (\ref{eq:3.1}), we know $h_{n,1}=O_P(\sqrt{\log m/n})$. By
(\ref{eq:3.7}) we have $h_{n,4}=O_P(1)$. Combining
(\ref{eq:3.truncation}) and the Bernstein's inequality, we can show
that
\begin{align*}
  h_{n,3}= O_P\left(\sqrt{{\log m}/{n}}\right).
\end{align*}
As an immediate consequence, we know $h_{n,2}=O_P(1)$. Therefore,
\begin{align*}
  \max_{1\leq i<j\leq m} \left|\check\tau_{n,i,j,1}-\check\tau_{n,i,j,2}\right| = O_P\left(\sqrt{{\log m}/{n}}\right),
\end{align*}
and (\ref{eq:3.8}) holds by using the assumption $\log m
=o(n^{q})=o(n^{1/3})$. The proof of Theorem~\ref{thm:max_cov} under
($\mathrm{A4'}$) is now complete.

\subsection{Proof under (A4)}
\label{sec:a4}

We follow the proof in Section~\ref{sec:a4'}, and point out necessary
modifications to make it work under (A4). If not specified, all the
notations have the same definitions as in Section~\ref{sec:a4'}. For
notational simplicity, we let $p=4(1+q)+\delta$.

\paragraph{\underline{\textnormal{\em Step 1: Truncation}}}

We truncate $X_{n,k,i}$ by
\begin{align*}
  \tilde X_{n,k,i} = X_{n,k,i}I\left\{|X_{n,k,i}|\leq n^{1/4}/\log n\right\},
\end{align*}
then
\begin{align*}
  P\left(\tilde M_n \neq M_n\right) \leq nm \mathcal{M}_n(p) n^{-p/4}(\log n)^p
  \leq C \mathcal{M}_n(p) n^{-\delta/4}(\log n)^p = o(1).
\end{align*}
Therefore, in the rest of the proof, it suffices to consider $\tilde
X_{n,k,i}$. For notational simplicity, we still use $\tilde X_{n,k,i}$
to denote its centered version with mean zero.

Elementary calculations show that
\begin{align}
  \label{eq:3.11}
  \max_{1\leq i \leq j \leq m} |\tilde \sigma_{n,i,j} - \sigma_{n,i,j}| & \leq C n^{-(p-2)/4} (\log n)^{p-2},
  \quad\hbox{and} \\ \label{eq:3.12}
  \max_{\alpha, \beta \in \mathcal{I}_n} \left|\Cov(\tilde X_{n,\alpha}, \tilde X_{n,\beta} ) -
  \Cov(X_{n,\alpha}, X_{n,\beta})\right| & \leq C n^{-(p-4)/4} (\log n)^{p-4}.
\end{align}
By (\ref{eq:3.11}), we know the covariance matrix of $(\tilde
X_{n,\alpha})_{\alpha\in\mathcal{I}_n}$ satisfies either (B1) or (B2)
if $\Sigma_n$ satisfies (B1) or (B2) correspondingly.
Since $$\E_0\tilde X_{n,\alpha} \in \mathscr{B}\left[1,8\sqrt{n}/(\log
  n)^2\right],$$ we know all the conditions of
Lemma~\ref{thm:max_mean} are satisfied, and hence (\ref{eq:3.4}) holds
for $M_{n,1}$.  Combining (\ref{eq:3.11}) and (\ref{eq:3.12}), we know
(\ref{eq:3.4}) also holds with if we replace $M_{n,1}$ by $M_{n,2}$.

\paragraph{\underline{\textnormal{\em Step 2: Effect of Estimated Means}}}

Using Bernstein's inequality, we can show
\begin{align*}
  \max_{1 \leq i \leq m}|\bar X_{n,i}| = O_P\left(\sqrt{\frac{\log n}{n}}\right),
\end{align*}
which implies that
\begin{align*}
  \left|M_{n,3} - M_{n,2}\right| = O_P\left(\frac{\log n}{n}\right) 
\end{align*}
and hence (\ref{eq:3.4}) also holds for $M_{n,3}$.

\paragraph{\underline{\textnormal{\em Step 3: Effect of Estimated Variances}}}
It suffices to show that
\begin{align}
  \label{eq:3.13}
  \max_{1\leq i<j\leq m} \left|\check\tau_{n,i,j}-\tilde\tau_{n,i,j}\right| = o_P(1/\log n).
\end{align}
Using (\ref{eq:3.4}), we know
\begin{align}
  \label{eq:3.15}
  \max_{1\leq i<j\leq m} \left|\check\tau_{n,i,j,1}-\check\tau_{n,i,j}\right| = O_P \left(\log n/n\right).
\end{align}
Since
\begin{align*}
  \left(\tilde X_{n,k,i} \tilde X_{n,k,j} - {\tilde\sigma_{n,i,j}}\right)^2 \leq 64 n/(\log n)^4.
\end{align*}
By Corollary~1.6 of \cite{nagaev:1979} (with $x=n/(\log n)^2$ and
$y=n/[2(\log n)^3]$ in their inequality (1.22)), we have
\begin{align*}
  \max_{1\leq i<j\leq m}
  P\left(|\check\tau_{n,i,j,2}-\tilde\tau_{n,i,j}| \geq (\log
    n)^{-2}\right) & \leq \left[\frac{Cn}{n(\log n)^{-2} \cdot [n(\log
      n)^{-3}/2]^{q\wedge 1}}\right]^{\log n} \\
  & \leq \left[\frac{C(\log n)^5}{n^{q\wedge 1}}\right]^{\log n},
\end{align*}
and it follows that
\begin{align}
  \label{eq:3.16}
  \max_{1\leq i<j\leq m} \left|\check\tau_{n,i,j,2}
  -\tilde\tau_{n,i,j}\right| = O_P\left[(\log n)^{-2}\right].
\end{align}
In view of (\ref{eq:3.15}), (\ref{eq:3.16}), we know to show
(\ref{eq:3.13}), it remains to prove
\begin{align}
  \label{eq:3.17}
  \max_{1\leq i<j\leq m} \left|\check\tau_{n,i,j,1}-\check\tau_{n,i,j,2}\right| = o_P(1/\log n).
\end{align}
We know $h_{n,1}=O_P(\sqrt{\log n/n})$ and $h_{n,4}=O_P(1)$. Using the
Bernstein's inequality, we can show that
\begin{align*}
  h_{n,3}= O_P\left(\sqrt{{\log n}/{n}}\right),
\end{align*}
and it follows that $h_{n,2}=O_P(1)$. Therefore,
\begin{align*}
  \max_{1\leq i<j\leq m} \left|\check\tau_{n,i,j,1}-\check\tau_{n,i,j,2}\right| = O_P\left(\sqrt{{\log n}/{n}}\right),
\end{align*}
and (\ref{eq:3.17}) holds. The proof of Theorem~\ref{thm:max_cov} under
($\mathrm{A4}$) is now complete.

\section{Some auxiliary results}
\label{sec:auxiliary} In this section we provide a normal
comparison principle and a Gaussian approximation result, and a
Poisson convergence theorem.

\subsection{A normal comparison principle}
\label{sec:normal_comparison}
Suppose for each $n\geq 1$, $(X_{n,i})_{i\in\mathcal{I}_n}$ is a
Gaussian random vector whose entries have mean zero and variance one,
where $\mathcal{I}_n$ is an index set with cardinality
$|\mathcal{I}_n|=s_n$. Let
$\Sigma_n=(r_{n,i,j})_{i,j\in\mathcal{I}_n}$ be the covariance matrix
of $(X_{n,i})_{i\in\mathcal{I}_n}$. Assume that $s_n\rightarrow\infty$
as $n\rightarrow\infty$.

We impose either of the following two conditions.
\begin{equation*}
  \begin{aligned}
    \hbox{({\bf B1}) } & \hbox{For any sequence $(b_n)$ such that $b_n\rightarrow\infty$, }
    \gamma(n,b_n)= o\left({1}/{\log b_n}\right);\\
    & \hbox{and } \limsup_{n\rightarrow\infty}\gamma_n<1. \\
    \hbox{({\bf B2}) } & \hbox{For any sequence $(b_n)$ such that $b_n\rightarrow\infty$, }
    \gamma(n,b_n)=o(1);\\
    & \sum_{i\neq j \in \mathcal{I}_n} r_{n,i,j}^2=O\left(s_n^{2-\delta}\right)
     \hbox{ for some } \delta>0; \hbox{ and }
    \limsup_{n\rightarrow\infty}\gamma_n<1.
  \end{aligned}
\end{equation*}
where
\begin{align*}
  & \gamma(n,b_n):=\sup_{i\in\mathcal{I}_n}
  \sup_{\mathcal{A}\subset\mathcal{I}_n,|\mathcal{A}|=b_n}
  \inf_{j\in \mathcal{A}}
  \left|r_{n,i,j}\right|\\
  & \hbox{and} \quad
  \gamma_n:=\sup_{i,j\in\mathcal{I}_n;\; i\neq j}|r_{n,i,j}|.
\end{align*}

\begin{lemma}\label{thm:normal_comparison}
  Assume either (B1) or (B2). For a positive real number $z_n$, define
  \begin{equation*}
    A_{n,i}'=\{|X_{n,i}| > z_n\} \quad\hbox{and}\quad
    Q_{n,d}' = \sum_{\mathcal{A}\subset\mathcal{I}_n,|\mathcal{A}|=d}
    P\left(\bigcap_{i\in\mathcal{A}} A_{n,i}'\right).
  \end{equation*}
  If $z_n$ satisfies that $z_n^2 = 2\log s_n - \log\log s_n -\log
  \pi + 2z + o(1)$, then for all $d \geq 1$.
  \begin{equation*}
  \lim_{n \to \infty} Q_{n,d}' = \frac{e^{-dz}}{d\,!},
\end{equation*}
\end{lemma}
Lemma~\ref{thm:normal_comparison} is a refined version of Lemma~20 in
\cite{wu:2011a}, so we omit the proof and put the details in a
supplementary file.

\begin{remark}
  \label{rk:empirical}
  The conditions imposed on $\gamma(n,b_n)$ seem a little involved. We
  have the following equivalent versions. Define
  \begin{align*}
    G_n(t) = \max_{i \in \mathcal{I}_n} \sum_{j\in\mathcal{I}_n}
    I\{|r_{n,i,j}|>t\}.
  \end{align*}
Then (i) $\gamma(n,b_n)=o(1)$ for any sequence $b_n \to \infty$ if
and only if the sequence $[G_n(t)]_{n\geq 1}$ is bounded for all
$t>0$; and (ii) $\gamma(n,b_n)(\log b_n)=o(1)$ for any sequence
$b_n \to \infty$ if and only if $G_n(t_n) = \exp \{o(1/t_n)\}$ for
any positive sequence $(t_n)$ converging to zero.
\end{remark}

\subsection{A Gaussian approximation result}
For a positive integer $d$, let $\mathfrak{B}_d$ be the Borel
$\sigma$-field on the Euclidean space $\mathbb{R}^d$. For two
probability measures $P$ and $Q$ on $\left(\mathbb{R}^d,
  \mathfrak{B}_d\right)$ and $\lambda>0$, define the quantity
\begin{equation*}
  \pi(P,Q;\lambda) =
  \sup_{A\in\mathfrak{B}_d}\left\{\max \left[P(A) - Q\left(A^\lambda\right),Q(A) - P\left(A^\lambda\right)\right]\right\},
\end{equation*}
where $A^{\lambda}$ is the $\lambda$-neighborhood of $A$
\begin{equation*}
  A^\lambda := \left\{x \in \mathbb{R}^d:\;\inf_{y\in A}|x-y|<\lambda\right\}.
\end{equation*}

For $\tau>0$, let $\mathscr{B}(d,\tau)$ be the collection of
$d$-dimensional random variables which satisfy the multivariate
analogue of the Bernstein's condition. Denote by $(x,y)$ the inner
product of two vectors $x$ and $y$.
\begin{equation}
  \label{eq:bernstein}
  \begin{aligned}
    \mathscr{B}(d,\tau)=& \left\{
    \xi \hbox{ is a random variable}:\;\E \xi=0, \hbox{ and }
      \phantom{\langle\xi,t\rangle^2}\right.\\
    & \left| \E \left[(\xi,t)^2(\xi,u)^{m-2}\right] \right|
      \leq \frac{1}{2}m!\tau^{m-2}\|u\|^{m-2}\E\left[(\xi,t)^2\right]  \\
    & \left.\hbox{for every }
    m=3,4,\ldots \hbox{ and for all } t,u\in\mathbb{R}^d
    \right\}.
  \end{aligned}
\end{equation}
The following Lemma on the Gaussian approximation is taken from
\cite{zaitsev:1987}.
\begin{lemma}
  \label{thm:zaitsev}
  Let $\tau>0$, and $\xi_1,\xi_2,\ldots,\xi_n \in \R^d$ be independent
  random vectors such that $\xi_i\in\mathscr{B}(d,\tau)$ for
  $i=1,2,\ldots,n$. Let $S=\xi_1+\xi_2+\ldots+\xi_n$, and
  $\mathscr{L}(S)$ be the induced distribution on $\R^d$. Let $\Phi$
  be the Gaussian distribution with the zero mean and the same
  covariance matrix as that of $S$. Then for all $\lambda>0$
  \begin{equation*}
    \pi[\mathscr{L}(S),\Phi;\lambda]
    \leq c_{1,d} \exp\left(-\frac{\lambda}{c_{2,d}\tau}\right),
  \end{equation*}
where the constants $c_{j,d},\;j=1,2$ may be taken in the form
$c_{j,d} = c_j d^{5/2}$.
\end{lemma}

\subsection{Poisson approximation: moment method}
\begin{lemma}
  \label{thm:poisson}
  Suppose for each $n\geq 1$, $(A_{n,i})_{i\in\mathcal{I}_n}$ is a
  finite collection of events. Let $I_{A_{n,i}}$ be the indicator
  function of $A_{n,i}$, and $W_n=\sum_{i\in\cal I}I_{A_{n,i}}$. For
  each $d \geq 1$, define
  \begin{equation*}
    Q_{n,d} = \sum_{\mathcal{A}\subset\mathcal{I}_n,|\mathcal{A}|=d}
    P\left(\bigcap_{i\in\mathcal{A}} A_{n,i}\right).
  \end{equation*}
  Suppose there exists a $\lambda>0$ such that
  \begin{equation*}
    \lim_{n\rightarrow\infty} Q_{n,d} = {\lambda^d}/{d\,!} \hbox{ for each } d\geq 1.
  \end{equation*}
  Then
  \begin{equation*}
    \lim_{n\rightarrow\infty} P(W_n = k) = \lambda^ke^{-\lambda}/k\,! \hbox{ for each } k\geq 0.
  \end{equation*}
\end{lemma}
Observe that for each $d\geq 1$, the $d$-th factorial moment of $W_n$
is given by
\begin{equation*}
  \E\left[W_n(W_n-1)\cdots(W_n-d+1)\right] = d\,! \cdot  Q_{n,d},
\end{equation*}
so Lemma~\ref{thm:poisson} is essentially the moment method. The
proof is elementary, and we omit details.


\renewcommand{\baselinestretch}{1}

\bibliographystyle{plainnat}
\bibliography{mybib}

\end{document}


\begin{frontmatter}

\title{Supplementary file of \\ Simultaneous Inference of Covariances}
\runtitle{Supplementary File}
\date{}

\begin{aug}
\author{\fnms{Han}
  \snm{Xiao}\ead[label=e1]{xiao@stat.rutgers.edu}}
\and
\author{\fnms{Wei Biao}
  \snm{Wu}\ead[label=e2]{wbwu@galton.uchicago.edu}}
\runauthor{H.~Xiao, W.B.~Wu}


\affiliation{Rutgers University \and University of Chicago}
\address{
501 Hill Center\\
110 Frelinghuysen Road\\
Piscataway, NJ 08854\\
\printead{e1}
}
\address{
Department of Statistics \\
5734 S.~University Ave\\
Chicago, IL  60637\\
\printead{e2}}
\end{aug}

\end{frontmatter}

In this document we give the proofs of Lemma~3, Lemma~4 and Lemma~7 of
the main article.

\begin{proof}[Proof of Lemma~3]
  Assume $X_i$ has mean zero and variance one. Let
  $\gamma_k=\E(X_0X_k)$ be the autocovariance of lag $k$. Then by
  Proposition~8,~Eq.~(34) of \cite{wu:2011a}, we know
  \begin{align}
    \label{eq:auto_cov} \seqla
    |\gamma_k| \leq \Psi_2\cdot\Psi_2(|k|).
  \end{align}
  \begin{enumerate}
  \item [(i)] Since $\Psi_4<\infty$, we know for any $\eta>0$, there
    exists a $N_1>0$ such that $|\gamma_k|<\eta$ when $k\geq N_1$. For
    $j\leq k$, define $\tilde
    X_{k,j}=g(\epsilon_k,\ldots,\epsilon_{j+1},\epsilon_j',\epsilon_{j-1}',\ldots)$,
    where $(\epsilon_i')_{i\in\Z}$ is an i.i.d. copy of
    $(\epsilon_i)_{i\in\Z}$.  By Eq. (38) of \cite{wu:2011a}, we know
    there exists a $N_2>0$ such that when $k\geq N_2$, $\|X_k-\tilde
    X_k\|_4 \leq \eta$. Set $N=\max\{N_1,N_2\}$, when $k \geq N$, we
    have
    \begin{align*}
      \Var(X_0X_k) & =\E(X_0^2X_k^2) - \gamma_k^2 =
      \E\left(X_k^2X_{k,j}^2\right) + \E\left[X_0^2(X_k^2-X_{k,j}^2)\right] - \gamma_k^2 \\
      & \geq 1 - \eta^2 - 2\|X_0\|_4^3\cdot\eta.
    \end{align*}
    Therefore, (A1) holds because $\eta$ can be arbitrarily small.
  \item [(ii)] We need to show that
    \begin{align*}
      \sup_{j\geq 0,\, 0\leq k \leq l,\, (0,j)\neq(k,l)}\Cor(X_0X_j,X_kX_l) < 1.
    \end{align*}
    It suffices to show that for some $N>0$
    \begin{align*}
      \sup_{j\geq 0,\, 0\leq k \leq l,\, (0,j)\neq(k,l),\, j+k+l\geq N}\Cor(X_0X_j,X_kX_l) < 1.
    \end{align*}
    If $j+k+l\geq N$, then the set $\{0,j,k,l\}$ can be partitioned
    into two non-empty subsets $\mathcal{B}_1$ and $\mathcal{B}_2$
    whose distance is no less than $N/6$. We only consider this type
    of partitions. If there is a partition such that one of
    $\mathcal{B}_1$ and $\mathcal{B}_2$ has cardinality one, then
    similarly as (i), we know for any $\eta>0$, when $N$ is large
    enough,
    \begin{align*}
      |\Cov(X_0X_j,X_kX_l)| = |\E(X_0X_jX_kX_l) - \gamma_j\gamma_{l-k}| \leq \eta.
    \end{align*}
    If for any partition both $\mathcal{B}_1$ and $\mathcal{B}_2$ has
    cardinality two, there are two sub-cases. (a) $j<k\leq l$ and
    $k-j\geq N/6$. For any $\eta>0$, when $N$ is large enough, we have
    \begin{align*}
      |\Cov(X_0X_j,X_kX_l)| = \left|\E\left[X_0X_j(X_kX_l-X_{k,j}X_{l,j})\right]\right| \leq \eta.
    \end{align*}
    (b) $\min\{j,l\}-k\geq N/6$. As in (i), for any $\eta>0$, when $N$
    is large enough, we have $\Var(X_0X_j)\geq 1- \eta$,
    $\Var(X_kX_l)\geq 1-\eta$, and $|\gamma_j\gamma_{l-k}|<\eta$. On
    the other hand, the condition $\Psi_4>0$ guarantees that the
    process is non-deterministic, and hence $\gamma:=\sup_{t\geq
      1}|\gamma_t|<1$. It follows that when $N$ is large enough
    \begin{align*}
      |\E(X_0X_jX_kX_l)| & = |\E(X_0X_{j,k}X_kX_{l,k}) + \E[X_0X_k(X_jX_l-X_{j,k}X_{l,k})]| \\
      & \leq \gamma + \eta.
    \end{align*}
    Therefore,
    \begin{align*}
      |\Cor(X_0X_j,X_kX_l)| \leq (\gamma + 2\eta)/(1-\eta) <1
    \end{align*}
    when $\eta$ is small enough. The proof of (ii) is now complete.
  \item [(iii)] We first consider ($\mathrm{A3}$). Note that
    \begin{align*}
      \Cov(X_iX_j,X_kX_l) & = \mathrm{Cum}(X_i,X_j,X_k,X_l) + \gamma_{i-k}\gamma_{j-l}+\gamma_{i-l}\gamma_{j-k},
    \end{align*}
    where $\mathrm{Cum}(X_i,X_j,X_k,X_l)$ is the fourth order joint
    cumulant of the random vector $(X_i,X_j,X_k,X_l)^\top$. Fix a
    subset $\{i,j\}$, for any integer $b>0$, there are at most $8b^2$
    subsets $\{k,l\}$ such that $\{k.l\}\subset B(i;b)\cup B(j;b)$,
    where $B(x;r)$ is the open ball $\{y:|x-y|<r\}$. For all other
    subsets $\{k,l\}$, by (\ref{eq:auto_cov}), we have
    \begin{align*}
      |\gamma_{i-k}\gamma_{j-l}+\gamma_{i-l}\gamma_{j-k}| \leq C\Psi_4(b).
    \end{align*}
    On the other hand, using similar arguments as Theorem~21 of
    \cite{wu:2011a}, we can show that
    \begin{align*}
      |\mathrm{Cum}(X_i,X_j,X_k,X_l)| \leq C\Psi_4(\floor{b/2}).
    \end{align*}
    Therefore, if $\Psi_4(k)=o(1/\log k)$ as $k\rightarrow\infty$,
    then ($\mathrm{A3}$) holds.

    Now we turn to ($\mathrm{A3'}$). Write
    $$\Cov(X_iX_j,X_kX_l)=\E(X_iX_jX_kX_l) -
    \gamma_{i-j}\gamma_{k-l}.$$ By (\ref{eq:auto_cov}), it is easily
    seen that
    \begin{align*}
      \sum_{1\leq i,j,k.l \leq m} \gamma_{i-j}^2\gamma_{k-l}^2 = O(m^{4-2\delta}).
    \end{align*}
    It then suffices to show
    \begin{align*}
      \sum_{1\leq i\leq j\leq k \leq l \leq m}[\E(X_iX_jX_kX_l)]^2 = O (m^{4-\delta}),
    \end{align*}
    which is true because by Eq. (38) of \cite{wu:2011a}
    \begin{align*}
      [\E(X_iX_jX_kX_l)]^2 = [\E(X_iX_jX_k(X_l-X_{l,k}))]^2 \leq 12\|X_0\|_4^6[\Psi_4(l-k)]^2.
    \end{align*}
  \end{enumerate}
  The proof of Lemma~3 is now complete.
\end{proof}

We now give the proof of Lemma~4.
\begin{proof}[Proof of Lemma~4]
  Suppose $(Y_1,Y_2,Y_3,Y_4)$ has a joint normal distribution. We can
  write $Y_i=\alpha_i^\top \h{Z}$, where $\h{Z}$ is a four dimensional
  standard Gaussian random vector. For any $0<\nu<1$, define the
  subset of $\R^{16}$,
  \begin{align*}
    D_{\nu}=\left\{(\alpha_1^\top,\alpha_2^\top,\alpha_3^\top,\alpha_4^\top):\,
      |\alpha_i|^2=1 \hbox{ and } |\alpha_i^\top\alpha_j|\leq 1-\nu \hbox{ for } 1 \leq i\neq j \leq 4.\right\}
  \end{align*}
  Since $|\Cor(Y_1Y_2,Y_3Y_4)|$ is a continuous function on $D_{\nu}$,
  and $D_{\nu}$ is compact, the maximum correlation is attained at
  some point in $D_\nu$.

  On the other hand, elementary calculation shows that
  $\Cor(Y_1Y_2,Y_3Y_4)=1$ if and only if $Y_1,Y_2,Y_3,Y_4$ are all
  perfectly correlated. The proof is now complete.
\end{proof}

The proof of Lemma~7 is a refined version of that of Lemma~20 in
\cite{wu:2011a}. We need the following bounds on normal tail
probabilities, which are taken from Lemma~19 of \cite{wu:2011a}.

Denote by $\varphi_d ((r_{ij}); x_1,\ldots, x_d)$ the density of a
$d$-dimensional multivariate normal random vector
$\h{X}=(X_1,\ldots,X_d)^\top$ with mean zero and covariance matrix
$(r_{ij})$, where we always assume $r_{ii}=1$ for $1 \leq i \leq d$
and $(r_{ij})$ is nonsingular. Let
$$Q_d\left((r_{ij});z\right) = \int_{z}^\infty
\cdots \int_{z}^{\infty}
\varphi_d\left((r_{ij}),x_1,\ldots,x_d\right) \dd x_{d} \cdots \dd
x_{1}.$$ 

\begin{lemmas}
  \label{thm:normbdd}
  For every $z>0$, $0<s<1$, $d \geq 1$ and $\epsilon>0$, there exists positive
  constants $C_d$ and $\epsilon_d$ such that for $0 < \epsilon < \epsilon_d$
  \begin{enumerate}
  \item if $|r_{ij}| < \epsilon$ for all $1 \leq i < j \leq d$, then
    \begin{align}
      \label{eq:2.normbdd3} \seqla
      Q_{d}\left((r_{ij});z\right) & \leq C_d \,f_d(\epsilon,1/z)\,
      \exp\left\{-\left(\frac{d}{2} - C_d \epsilon
        \right)z^2\right\} 
    \end{align}
    where $f_{2k}(x,y)=\sum_{l=0}^k x^ly^{2(k-l)}$ and
    $f_{2k-1}(x,y)=\sum_{l=0}^{k-1}x^ly^{2(k-l)-1}$ for $k \geq 1$;
  \item if for all $1 \leq i < j \leq d+1$ such that $(i,j)\neq(1,2)$,
    $|r_{ij}| \leq \epsilon$, then
    \begin{equation}
      \label{eq:2.normbdd4} \seqla
      Q_{d+1}\left((r_{ij});z\right) \leq C_d
      \exp\left\{-\left(\frac{(1-|r_{12}|)^2+d}{2}- C_d\epsilon\right)z^2\right\}.
    \end{equation}
  \end{enumerate}
\end{lemmas}

We first give a one-sided version of Lemma~7 and its proof, then we
show how it implies Lemma~7.

\begin{lemmas}
  \label{thm:normal_comparison_2}
  Assume either (B1) or (B2). For a positive real number $z_n$, define
  the event $A_{n,i}$ and $Q_{n,d}$ as 
  \begin{equation*}
    A_{n,i}=\{X_{n,i} > z_n\} \quad\hbox{and}\quad
    Q_{n,d} = \sum_{\mathcal{A}\subset\mathcal{I}_n,|\mathcal{A}|=d}P\left(\bigcap_{i\in\mathcal{A}} A_{n,i}\right).
  \end{equation*}
  If $z_n$ satisfies that $z_n^2 = 2\log s_n - \log\log s_n -\log
  (4\pi) + 2z + o(1)$, then for all $d \geq 1$
  \begin{equation*}
    \lim_{n \to \infty} Q_{n,d} = \frac{e^{-dz}}{d\,!}.
\end{equation*}
\end{lemmas}


\begin{proof}
  The following facts about normal tail probabilities are well-known:
  \begin{equation}
    \label{eq:2.normbdd} \seqla
    P(X_1\geq x) \leq \frac{1}{\sqrt{2\pi}x} e^{-x^2/2} \hbox{ for } x>0
    \quad \hbox{and} \quad
    \lim_{x \to \infty}\frac{P(X_1 \geq x)}{(1/x)(2\pi)^{-1/2}
    \exp\left\{-x^2/2\right\}}=1,
  \end{equation}
  By the assumption on $z_n$, if for each $n$,
  $X_{n,i},\,i\in\mathcal{I}_n$ are i.i.d., then by
  (\ref{eq:2.normbdd}),
  \begin{eqnarray*}
    \lim_{n \to \infty} Q_{n,d}
     &=& \lim_{n\to\infty} {n
    \choose d} Q_d(I_d,z_n) \cr
    &=& \lim_{n \to \infty} {n \choose d}
    \frac{1}{(2\pi)^{d/2}z_n^d}
          \exp\left\{-\frac{dz_n^2}{2}\right\}
     = \frac{e^{-dz}}{d!}.
  \end{eqnarray*}
  When the $X_{n,i}$'s are dependent, the result is still trivially
  true when $d=1$. Now we deal with the $d\geq 2$ case. Suppose
  $(b_n)$ is a sequence of positive numbers which converges to
  infinity. For each subset $J$ of $\mathcal{I}_n$ with cardinality
  $|J|=d$, we define an undirected graph $\mathscr{G}(J)$ by
  identifying each $i\in J$ with a node and saying $i$ and $j$ are
  adjacent if $|r_{n,i,j}|>\gamma(n,b_n)$. Suppose the graph
  $\mathscr{G}(J)$ has $d-s$ connected components
  $\mathcal{B}_1,\ldots,\mathcal{B}_{d-s}$.
  If $s \geq 1$, assume w.l.o.g.~that $|\mathcal{B}_1|\geq 2$. Pick
  $k_0,k_1 \in \mathcal{B}_1$, and $k_p \in \mathcal{B}_p$ for $2\leq
  p \leq d-s$, and set $K=\{k_0,k_1,k_2,\ldots,k_{d-s}\}$. Define
  $Q_J=P(\cap_{k \in J}A_k)$ and $Q_K$ similarly, then $Q_J \leq
  Q_K$. By (\ref{eq:2.normbdd4}) of Lemma~\ref{thm:normbdd}, there
  exists a number $M>1$ depending on $d$ and the sequences
  $(\gamma_n)$ and $(b_n)$, such that when $n\geq M$,
  \begin{align*}
    Q_K &\leq C_{d-s} \exp\left\{-\left(\frac{(1-\gamma_n)^2+d-s}{2}- C_{d-s}\gamma(n,b_n)\right)z_n^2\right\} \cr
    &\le C_{d-s}\exp\left\{-\left(\frac{d-s}{2} + \frac{(1-\gamma_n)^2}{3} \right) z_n^2\right\}.
  \end{align*}
  Note that $z_n^2 = 2\log s_n - \log\log s_n + O(1)$. Pick
  $b_n=\floor{s_n^{\alpha}}$ for some $\alpha <(1-\gamma_n)^2/3d$. For
  any $1 \leq a \leq d-1$, since there are at most $O\left(b_n^a
    s_n^{d-a}\right)$ subsets $J \subset \mathcal{I}_n$ such that
  $|J|=d$ and the graph $\mathscr{G}(L)$ has $d-a$ connected
  components, we know the sum of $Q_J$ over these $J$ is dominated by
  $$C_{d-a}\exp\left\{\log s_n \left( (d-a) + \frac{2(d-1)(1-\gamma_n)^2}{3d}
      -(d-a) - \frac{2(1-\gamma_n)^2}{3} \right)\right\}$$ when $n$ is
  large enough, which converges to zero. Therefore, it remains to
  consider all the subsets $J\subset \mathcal{I}_n$ such that the
  graph $\mathscr{G}(J)$ has no edges

  Let $J \subset \mathcal{I}_n$ be a subset such that $|J|=d$, and
  $\left|r_{n,i,j}\right|<\gamma(n,b_n)$ for all pairs $i,j$ such that
  $i,j\in J$ and $i\neq j$, and $\mathcal{J}(d,b_n)$ be the collection
  of all such subsets. Let $(r_{ij})_{i,j\in J}$ be the
  $d$-dimensional covariance matrix of $\h{X}_J := (X_{n,i})_{i \in
    J}$.  There exists a matrix $R_J=\theta(r_{ij})_{i,j\in
    J}+(1-\theta)I_d$ for some $0 < \theta < 1$ such that
  \begin{equation*}
    Q_J-Q_d(I_d,z_n) = \sum_{h,l \in J, h<l} \frac{\partial
      Q_d}{\partial r_{hl}}[R_J;z_n]r_{hl}.
  \end{equation*}
Let $R_H$, $H = J \setminus \{h,l\}$, be the correlation matrix of
the conditional distribution of $\h{X}_H$ given $X_h$ and $X_l$.
By (\ref{eq:2.normbdd3}) of Lemma~\ref{thm:normbdd}, for $n$ large
enough
  \begin{align*}
    \frac{\partial Q_{d}}{\partial r_{hl}}[R_J;z_n] & \leq C
    \exp\left\{-\frac{z_n^2}{1+|r_{n,h,l}|}\right\} \cdot
    Q_{d-2}\left(R_K;(1-3\gamma(n,b_n))z_n\right) \\
    & \leq C C_{d-2}f_{d-2}(\gamma(n,b_n),1/z_n) 
    \exp\left\{-\frac{z_n^2}{1+|r_{n,h,l}|}\right\} \cr
    & \quad \times
      \exp\left\{-\left(\frac{d-2}{2} - 2C_{d-2} \gamma(n,b_n)
        \right)(1-3\gamma(n,b_n))^2z_n^2\right\} \\
    & \leq C_df_{d-2}(\gamma(n,b_n),1/z_n) \\
    &\quad \times \exp\left\{-\left(\frac{d}{2} - (2C_{d-2}+3(d-2))
        \gamma(n,b_n) -|r_{n,h,l}| \right)z_n^2\right\}\\
    & \leq C_df_{d-2}(\gamma(n,b_n),1/z_n)
         \exp\left\{-\left(\frac{d}{2} - C_d
        \gamma(n,b_n) \right)z_n^2\right\}.
  \end{align*}
  It follows that
  \begin{equation}
  \label{eq:2.5} \seqla
  \begin{aligned}
  & \sum_{J \in \mathcal{J}(d,b_n)} |Q_J-Q_d(I_d;z_n)| \cr 
  & \leq C_d f_{d-2}(\gamma(n,b_n),1/z_n) \cr 
  & \quad \times \sum_{J \in \mathcal{J}(d,b_n)} \sum_{i,j\in J;\, i \neq j}
  \exp\left\{-\left(\frac{d}{2} - C_d \gamma(n,b_n) \right)z_n^2\right\} |r_{n,i,j}| \cr 
  & \leq C_d f_{d-2}(\gamma(n,b_n),1/z_n) s_n^{d-2} \cr 
  & \quad \times \sum_{i,j\in \mathcal{I}_n}^\ast 
  \exp\left\{-\left(\frac{d}{2} - C_d \gamma(n,b_n) \right)z_n^2\right\} |r_{n,i,j}|, 
  \end{aligned}     
  \end{equation}
  where the sum $\sum_{i,j\in\mathcal{I}_n}^{\ast}$ is over all the
  pair $(i,j)$ such that $|r_{n,i,j}|\leq\gamma(n,b_n)$.  Under the
  assumption (B1), we have
  \begin{equation}
    \label{eq:2.6} \seqla
    \begin{aligned}
      & \sum_{J \in \mathcal{J}(d,b_n)} |Q_J-Q_d(I_d;z_n)| \cr 
      & \leq C_d f_{d-2}(\gamma(n,b_n),1/z_n) (\log s_n)^{d/2} \gamma(n,b_n) 
      \exp\left\{  C_d\gamma(n,b_n)(\log s_n)\right\}
    \end{aligned}
  \end{equation}
  Since $\lim_{n \to \infty} \gamma(n,b_n) \log b_n =0$, it also holds
  that $\lim_{n \to \infty} \gamma(n,b_n) \log s_n =0$. Note that
  $\lim_{n\to\infty} (\log s_n)^{1/2}/z_n = 2^{-1/2}$, it follows that
  $\lim_{n\to\infty} f_{d-2}(\gamma(n,b_n),1/z_n) (\log
  s_n)^{d/2-1} = 2^{-d/2+1}$. Therefore, the term in (\ref{eq:2.6})
  converges to zero, and the theorem holds under (B1).

  Alternatively, if (B2) is true, from (\ref{eq:2.5}) we have
  \begin{equation*}
    \begin{aligned}
      & \sum_{J \in \mathcal{J}(d,b_n)} |Q_J-Q_d(I_d;z_n)| \cr 
      & \leq C_d f_{d-2}(\gamma(n,b_n),1/z_n) s_n^{-2} (\log s_n)^{d/2} 
      \sum_{i,j\in\mathcal{I}_n}^{\ast} \exp\left\{  C_d\gamma(n,b_n)(\log s_n)\right\}|r_{n,i,j}| \cr
      & \leq C_d f_{d-2}(\gamma(n,b_n),1/z_n) s_n^{-1} (\log s_n)^{d/2} \exp\left\{  C_d\gamma(n,b_n)(\log s_n)\right\}
      \left(\sum_{i,j\in\mathcal{I}_n} r_{n,i,j}^2\right)^{1/2} \cr
      & \leq C_d  s_n^{-\delta/2} (\log s_n) \exp\left\{  C_d\gamma(n,b_n)(\log s_n)\right\} = o(1),
    \end{aligned}
  \end{equation*}
  and the proof is complete.
\end{proof}

Now we give the proof of Lemma~7.
\begin{proof}[Proof of Lemma~7]
  In the proof of Theorem~\ref{thm:normal_comparison_2}, the upper
  bounds on $Q_J$ and $|Q_J-Q(I_d;z_n)|$ are expressed through the
  absolute values of the covariances, so we can obtain the same bounds
  for probabilities of the form $P(\cap_{1\leq i \leq d}
  \{(-1)^{a_i}X_{t_i} \geq z_n\})$ for any $(a_1,\ldots,a_d) \in
  \{0,1\}^d$. Based on this observation, Lemma~7 is an immediate
  consequence of Lemma~\ref{thm:normal_comparison_2}.

\end{proof}

\renewcommand{\baselinestretch}{1}

\bibliographystyle{plainnat}
\bibliography{mybib}